\documentclass[12pt]{amsart}
\usepackage{enumerate, hyperref}
\usepackage{xy}
\xyoption{all}
\theoremstyle{plain}

\newtheorem*{theorem*}{Theorem}

\newtheorem*{prop*}{Proposition}

\newtheorem*{claim*}{Claim}

\newtheorem*{cor*}{Corollary}
\newtheorem*{lemma*}{Lemma}

\theoremstyle{definition}

\newtheorem*{defn*}{Definition}

\newtheorem*{remark*}{Remark}

\newcommand{\qbin}[2]{\binom{#1}{#2}_q}

\newcommand{\C}{\mathbf C}
\newcommand{\dash}{\nobreakdash-}

\newcommand{\inv}{\mathrm{inv}}

\newcommand{\multinom}[2]{\binom #1{\begin{matrix}#2\end{matrix}}}

\newcommand{\Z}{\mathbf Z}

\title{Counting subspaces of a finite vector space}
\author{Amritanshu Prasad}
\subjclass[2000]{05A10, 05A30}
\keywords{Gaussian binomial coefficients, finite vector spaces}
\address{The Institute of Mathematical Sciences, Chennai.}

\begin{document}
\begin{abstract}
  We discuss the relation between the Gaussian binomial and multinomial coefficients and ordinary binomial and multinomial coefficients from a combinatorial viewpoint, based on expositions by Knuth, Stanley and Butler.
\end{abstract}

\maketitle

\subsection{Background}
\label{sec:background}
The Gaussian binomial coefficients were introduced two hundred years ago by Carl Friedrich Gauss \cite{summatio-quarumdam} as a tool to find a formula for the following sums, now called Gauss sums:
\begin{equation*}
  G=\sum_{n=0}^{p-1} e^{2\pi i n^2/p} \quad \text{for all primes } \, p.
\end{equation*}
Gauss had easily computed these sums up to sign in 1801, but it took him four years of intense effort to resolve the sign ambiguity; he solved the problem only in 1805. This aspect of Gaussian binomial coefficients makes for a fascinating story in itself, but we shall simply refer the reader to \cite{BerndtEvans}, where an exposition of Gauss's proof (among other things) can be found.

This article concentrates on a different aspect of Gaussian binomial coefficients: they arise as the answer to a counting problem in linear algebra involving finite fields. Vectors with real coordinates arise naturally in our efforts to describe nature by mathematics. Space is described by three real coordinates, space-time by four. Physicists need to keep track of the position and momentum of a particle in space in order to be able to predict is position in the future, leading to the use of six real coordinates for each particle.

Number theorists like to work with vectors with integer or rational coordinates: prospective solutions to Fermat's famous equation
\begin{equation}
  \label{eq:18}
  x^n + y^n = z^n
\end{equation}
are vectors with three integer or rational coordinates.

Finding out whether integer solutions exist to a given Diophantine equations can be  notoriously difficult, as evidenced by all the fuss over (\ref{eq:18}). The most powerful weapon in the mathematician's arsenal to tackle this problem was also introduced by Gauss: the notion of \emph{modular arithmetic}, which is the arithmetic of remainders.
Four hours after 9 o'clock, it is 1 o'clock, because the remainder when 9+4=13 is divided by 12 is 1.
The clock does arithmetic modulo 12, but one can do it modulo any number $n\geq 2$. In arithmetic modulo $n$, two integers $a$ and $b$ are identified if their remainders after division by $n$ are the same; or, as number theorists say it, their residues modulo $n$ are the same. We write
\begin{equation*}
  a\equiv b \mod n
\end{equation*}
Thus, there are $n$ residue classes of integers modulo $n$, represented by
\begin{equation*}
  0,1,2,\ldots,n-1.
\end{equation*}
Addition and multiplication carry over to modular arithmetic and satisfy the same rules of commutativity, associativity and distributivity that hold for the integers.

When $n$ is a prime number (call it $p$), something special happens: it is possible to divide by any non-zero residue modulo $p$, meaning to say that for any $x$ not divisible by $p$, there exists an integer $y$ such that
\begin{equation*}
  xy\equiv 1\mod p.
\end{equation*}
This integer $y$ can be thought of as the reciprocal of $x$ modulo $p$, and we may write it as $1/x$.
A number system, namely a set with the operations of addition and multiplication, which satisfies the usual axioms of commutativity, associativity and distributivity (it is assumed that the set comes with a \lq $0$\rq{} and a \lq $1$\rq{}), where division by any element different from $0$ is possible, is called a \emph{field}.
For every prime power $q$, \'Evariste Galois (1811--1832) constructed a finite field with $q$ elements \cite{Galois-fields}.
In 1903, Eliakim Hastings Moore showed that for each prime power $q$, all fields with $q$ elements are isomorphic \cite{Moore-finite-fields}.
In other words there is, in essence, just one field with $q$ elements for each prime power $q$.

Vectors with coordinates in any field behave very much like vectors with real numbers as coordinates.
The usual notions of linear independence, basis, and subspace carry over from vectors with real coordinates to vectors with coordinates in finite fields, and will be used freely throughout this article.
We recall the essentials: a set of vectors which is closed under addition and scalar multiplication (a scalar is an element of the underlying field) is called a \emph{linear subspace}. A collection $v_1,\ldots,v_n$ of vectors is said to be \emph{linearly independent} if, for each $1\leq i<n$, there exists a linear subspace which contains $v_1,\ldots,v_{i-1}$ but not $v_i$.
A \emph{basis} of a linear subspace is any maximal linearly independent subset.
Every vector in the subspace can be written as a sum of scalar multiples of elements from a basis.
All bases of a linear subspace have the same number of elements, and this number is called the \emph{dimension} of the linear subspace.

\subsection{Introduction}
\label{sec:introduction}

Let $F$ be a finite field of order $q$.
Let $V=F^n$ be the space of vectors with $n$ coordinates in $F$.
For every $k\leq n$, define $\qbin n k$ as the number of linear subspaces of $V$ of dimension $k$.
The number $\qbin n k$ is called a Gaussian binomial coefficient.
One of the goals of this article is to explore the relationship between $\qbin n k$ and the binomial coefficient $\binom n k$, which is the number of ways of choosing $k$ objects out of $n$ ($n$ choose $k$).

It is not difficult to write down a formula for $\qbin n k$.
To do this, note that a $k$\dash dimensional subspace is specified by giving $k$ linearly independent vectors $\{v_1,\ldots,v_k\}$ in $V$.
In how many ways can this be done?
Firstly, $v_1$ can be taken to be any non-zero vector in $V$.
Therefore there are $q^n-1$ choices for $v_1$.
Given $v_1$, $v_2$ can be chosen to be any vector which is not in the subspace spanned by $v_1$.
Since this subspace has $q$ elements, there are $q^n-q$ choices for $v_2$.
Continuing in this manner, we see that given $v_1,\ldots,v_l$ for $l<k$, there are $q^n-q^l$ choices for $v_{l+1}$.
The number of sets of $k$ linearly independent vectors in $V$ is therefore
\begin{equation*}
  (q^n-1)(q^n-q)\cdots (q^n-q^{k-1}).
\end{equation*}
Applying the above formula to the special case where $n=k$, we see that each $k$\dash dimensional subspace of $V$ (which is a $k$\dash dimensional vector space in its own right) has
\begin{equation*}
  (q^k-1)(q^k-q)\cdots (q^k-q^{k-1})
\end{equation*}
bases.
To obtain the number of $k$\dash dimensional subspaces, we divide the first expression by the second one.
Thus the number of $k$\dash dimensional subspaces of $V$ is
\begin{equation}
  \label{eq:5}
  \qbin n k=\frac{(q^n-1)(q^n-q)\cdots (q^n-q^{k-1})}{(q^k-1)(q^k-q)\cdots (q^k-q^{k-1})}.
\end{equation}
So far, $q$ has been the order of a finite field (which can be any prime power), but the above expression is also a rational function of $q$ with a denominator that vanishes only at $q=1$.
This allows us to get a value for $\qbin n k$ for any complex number $q\neq 1$.
Using L'H\^opital's rule one computes
\begin{equation*}
  \lim_{q\to 1} \frac{q^n-q^i}{q^k-q^i}=\frac{n-i}{k-i}
\end{equation*}
for $i=0,\ldots,n-1$.
Therefore
\begin{equation}
  \label{eq:3}
  \lim_{q\to 1} \qbin n k= \frac{n(n-1)\cdots(n-k+1)}{k(k-1)\cdots 1}=\binom n k.
\end{equation}

\subsection{Counting using echelon form}
\label{sec:echelon}

In 1971, Donald Ervin Knuth provided an elegant explanation for the identity (\ref{eq:3}) in \cite{MR0270933} based on the idea that every subspace has a unique basis which is in reduced row echelon form.

Let $F$ be a field with $q$ elements.
Any $k$ linearly independent vectors in $F^n$ can be arranged into a $k\times n$ matrix of rank $k$, where the entries of the $i$th row are the coordinates of the $i$th vector.
In other words, every $k$\dash dimensional subspace of $F^n$ is the row space of an $k\times n$ matrix of rank $k$.
The row space of such a matrix does not change under the following elementary row operations:
\begin{enumerate}
\item permutation of the rows
\item addition of a scalar multiple of one row to another
\item multiplication of a row by a non-zero scalar
\end{enumerate}
A $k\times n$ matrix is said to be in reduced row echelon form if
\begin{enumerate}
\item the left-most non-zero entry of each row is $1$ (let's call it a \lq\lq leading $1$\rq\rq)
\item all the other entries in the column of a leading $1$ are zero
\item the leading $1$ in any row occurs to the right of the leading 1 in the row above it
\end{enumerate}
A matrix in reduced row echelon form looks something like
\begin{equation*}
  \begin{pmatrix}
    0 \cdots 0 & 1 & * \cdots * & 0 & * \cdots * & 0 & * \cdots *\\
    0 \cdots 0 & 0 & 0 \cdots 0 & 1 & * \cdots * & 0 & * \cdots *\\
    0 \cdots 0 & 0 & 0 \cdots 0 & 0 & 0 \cdots 0 & 1 & * \cdots *\\
    \vdots & \vdots & \vdots & \vdots & \vdots & \vdots & \vdots 
  \end{pmatrix}
\end{equation*}
where the $*$'s represent arbitrary elements of $F$.

The reader should not have much difficulty in seeing that the elementary row operations described above can be used to reduce any $k\times n$ matrix to reduced row echelon form.
If the matrix has rank $k$, then all the rows in the reduced row echelon form are non-zero.
Moreover, if two matrices in reduced row echelon form have the same row space, they are equal.

Given a matrix in reduced row echelon form, delete all the entries in each row to the left of the leading $1$.
Then remove the columns containing the leading $1$'s.
Finally, replace each remaining entry with a $*$.
The result is a pattern of $*$'s inside a grid with $k$ rows and $n-k$ columns.
These patterns are characterized by the property that except for the $*$'s in the first row, every $*$ has a $*$ above it, and except for the $*$'s in the $(n-k)$th column, every $*$ has a $*$ to the right of it.
They are called Ferrers diagrams.
\begin{figure}
  \centering
  \begin{equation*}
    \begin{pmatrix}
      0 & 1 & 1 & 0 & 0 & 1 & 2\\
      0 & 0 & 0 & 1 & 0 & 2 & 0\\
      0 & 0 & 0 & 0 & 1 & 0 & 1
    \end{pmatrix}
    \quad
    \longrightarrow
    \quad
      \begin{xy}
        (0,7.5);
        (5,7.5)**@{.};(10,7.5)**@{.};(15,7.5)**@{.};(20,7.5)**@{.};
        (0,2.5);(5,2.5)**@{.};(10,2.5)**@{.};(15,2.5)**@{.};(20,2.5)**@{.};
        (0,-2.5);(5,-2.5)**@{.};(10,-2.5)**@{.};(15,-2.5)**@{.};(20,-2.5)**@{.};
        (0,-7.5);(5,-7.5)**@{.};(10,-7.5)**@{.};(15,-7.5)**@{.};(20,-7.5)**@{.};
        (0,7.5);(0,2.5)**@{.};(0,-2.5)**@{.};(0,-7.5)**@{.};
        (5,7.5);(5,2.5)**@{.};(5,-2.5)**@{.};(5,-7.5)**@{.};
        (10,7.5);(10,2.5)**@{.};(10,-2.5)**@{.};(10,-7.5)**@{.};
        (15,7.5);(15,2.5)**@{.};(15,-2.5)**@{.};(15,-7.5)**@{.};
        (20,7.5);(20,2.5)**@{.};(20,-2.5)**@{.};(20,-7.5)**@{.};
        (7.5,5)*=0{*};(12.5,5)*=0{*};(17.5,5)*=0{*};
        (12.5,0)*=0{*};(17.5,0)*=0{*};
        (12.5,-5)*=0{*};(17.5,-5)*=0{*};
      \end{xy}
  \end{equation*}
  \caption{From a matrix to its Ferrers diagram}
  \label{fig:2}
\end{figure}
Figure~\ref{fig:2} illustrates an example of this process when $n=7$ and $k=3$.
On the left is a matrix in reduced row echelon form.
The leading $1$'s occur in the second, fourth and fifth columns.
After deleting the entries to the left of the leading $1$'s, deleting the columns with leading $1$'s, and replacing the remaining entries with $*$'s the Ferrers diagram on the right is obtained.

The above process can be easily reversed: let $e_1,\ldots,e_k$ denote the $k$ coordinate vectors in $F^n$, written a columns.
Starting with a Ferrers diagram $\lambda$ in a $k\times (n-k)$ grid, replace each $*$ with an element of $F$ (this can be done in $q^{|\lambda|}$ ways, where $|\lambda|$ denotes the number of $*$'s in $\lambda$) and each vacant square in the grid by a $0$.
Insert $e_1,\ldots,e_k$ in reverse order as follows:
insert $e_k$ to the left of the left-most $*$ in the bottom row.
Having inserted $e_{i+1},\ldots,e_k$, insert $e_i$ to the left of all the $*$'s in the $i$th row and all the columns $e_{i+1},\ldots,e_n$.
This gives all the matrices associated with $\lambda$.

The problem of counting the $k$\dash dimensional subspaces of $F^n$ is equivalent to counting the number of $k\times n$ matrices of rank $k$ in reduced row echelon form.
Each matrix in reduced row echelon form gives rise to a Ferrers diagram inside a $k\times(n-k)$ grid.
Each Ferrers diagram $\lambda$ can be obtained from $q^{|\lambda|}$ matrices in reduced row echelon form.
We have shown that
\begin{equation}
  \label{eq:4}
  \qbin n k = \sum_{\lambda\subset k\times n-k} q^{|\lambda|}
\end{equation}
where $\lambda\subset k\times n-k$ is supposed to indicate that $\lambda$ is a Ferrers diagram in a grid with $k$ rows and $n-k$ columns.

The expression (\ref{eq:4}) for $q$-binomial coefficients demonstrates that $\qbin n k$ is a monic polynomial in $q$ of degree $k(n-k)$ with positive integer coefficients, which is not evident from (\ref{eq:5}).
In the next section, we take a closer look at the right hand side of (\ref{eq:4}) to give, among other things, another proof of (\ref{eq:3}).

A partition of a natural number $n$ is a decomposition
\begin{equation*}
  n=n_1+\cdots+n_k,
\end{equation*}
where each part $n_i$ is positive.
If one does not distinguish between different reorderings of the same summands, one may assume that $n_1\geq \ldots\geq n_k$.
A Ferrers diagram $\lambda$ with $|\lambda|=n$ can be thought of as a partition of $n$, the parts being the number of $*$'s in each row of $\lambda$.
Thus, (\ref{eq:4}) has the following interpretation:
\begin{quote}
  The coefficient of $q^r$ in $\qbin n k$ is the number of partitions of $r$ into no more than $n-k$ parts, with each part being no larger than $k$.
\end{quote}

\subsection{Paths, subsets, and permutations}
\label{sec:paths-subs-perm}

Given a Ferrers diagram $\lambda \subset k\times n-k$, we identify it with a path from the top-left corner to the bottom-right corner of a rectangular grid of squares of height $k$ and length $n-k$ as follows:
begin at the top left corner. As each stage, if you find yourself at the top left corner of a square with a $*$ in it, move one step downwards. If not, move one step to the right.
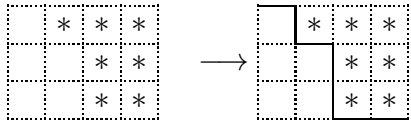
\begin{figure}[b]
  \centering
  \begin{equation*}
      \begin{xy}
        (0,7.5);
        (5,7.5)**@{.};(10,7.5)**@{.};(15,7.5)**@{.};(20,7.5)**@{.};
        (0,2.5);(5,2.5)**@{.};(10,2.5)**@{.};(15,2.5)**@{.};(20,2.5)**@{.};
        (0,-2.5);(5,-2.5)**@{.};(10,-2.5)**@{.};(15,-2.5)**@{.};(20,-2.5)**@{.};
        (0,-7.5);(5,-7.5)**@{.};(10,-7.5)**@{.};(15,-7.5)**@{.};(20,-7.5)**@{.};
        (0,7.5);(0,2.5)**@{.};(0,-2.5)**@{.};(0,-7.5)**@{.};
        (5,7.5);(5,2.5)**@{.};(5,-2.5)**@{.};(5,-7.5)**@{.};
        (10,7.5);(10,2.5)**@{.};(10,-2.5)**@{.};(10,-7.5)**@{.};
        (15,7.5);(15,2.5)**@{.};(15,-2.5)**@{.};(15,-7.5)**@{.};
        (20,7.5);(20,2.5)**@{.};(20,-2.5)**@{.};(20,-7.5)**@{.};
        (7.5,5)*=0{*};(12.5,5)*=0{*};(17.5,5)*=0{*};
        (12.5,0)*=0{*};(17.5,0)*=0{*};
        (12.5,-5)*=0{*};(17.5,-5)*=0{*};
      \end{xy}
      \quad 
      \longrightarrow
      \begin{xy}
        (0,7.5);
        (5,7.5)**@{-};(10,7.5)**@{.};(15,7.5)**@{.};(20,7.5)**@{.};
        (0,2.5);(5,2.5)**@{.};(10,2.5)**@{-};(15,2.5)**@{.};(20,2.5)**@{.};
        (0,-2.5);(5,-2.5)**@{.};(10,-2.5)**@{.};(15,-2.5)**@{.};(20,-2.5)**@{.};
        (0,-7.5);(5,-7.5)**@{.};(10,-7.5)**@{.};(15,-7.5)**@{-};(20,-7.5)**@{-};
        (0,7.5);(0,2.5)**@{.};(0,-2.5)**@{.};(0,-7.5)**@{.};
        (5,7.5);(5,2.5)**@{-};(5,-2.5)**@{.};(5,-7.5)**@{.};
        (10,7.5);(10,2.5)**@{.};(10,-2.5)**@{-};(10,-7.5)**@{-};
        (15,7.5);(15,2.5)**@{.};(15,-2.5)**@{.};(15,-7.5)**@{.};
        (20,7.5);(20,2.5)**@{.};(20,-2.5)**@{.};(20,-7.5)**@{.};
        (7.5,5)*=0{*};(12.5,5)*=0{*};(17.5,5)*=0{*};
        (12.5,0)*=0{*};(17.5,0)*=0{*};
        (12.5,-5)*=0{*};(17.5,-5)*=0{*};
      \end{xy}
  \end{equation*}
  \caption{From a Ferrers diagram to its path}
  \label{fig:fp}
\end{figure}
In the example with $n=7$ and $k=3$ that we considered in Figure~\ref{fig:2}, this process is illustrated in Figure~\ref{fig:fp}.
Such a path always consists of $n$ segments of unit length, of which $k$ are vertical and $n-k$ are horizontal.
Index the segments of such a path by the numbers $1,\ldots,n$, starting at the top-left corner and ending at the bottom-right corner.
The path is completely determined by specifying which $k$ of these $n$ segments are vertical.
Figure~\ref{fig:1} shows all ten paths for $n=5$ and $k=2$.
Directly beneath each path is listed the subset of $\{1,2,3,4,5\}$ corresponding to the vertical segments.

Since the paths are completely specified by which $k$ of the $n$ segments are vertical, the number of such paths is $\binom n k$.
This gives a more illuminating proof of the identity (\ref{eq:3}):
\begin{equation*}
  \lim_{q\to 1}\qbin n k = \lim_{q\to 1} \sum_{\lambda\subset k\times n-k} q^{|\lambda|} = \sum_{\lambda\subset k\times n-k} 1 = \binom n k.
\end{equation*}
This proof trumps the one in Section~\ref{sec:introduction} by virtue of being combinatorial; in effect, it constructs a surjective function $K$ (call it a collapse) from the set of $k$\dash dimensional subspaces of $F^n$ to the set of subsets of $\{1,\ldots,n\}$ of order $k$ with all pre-images having cardinality a power of $q$.
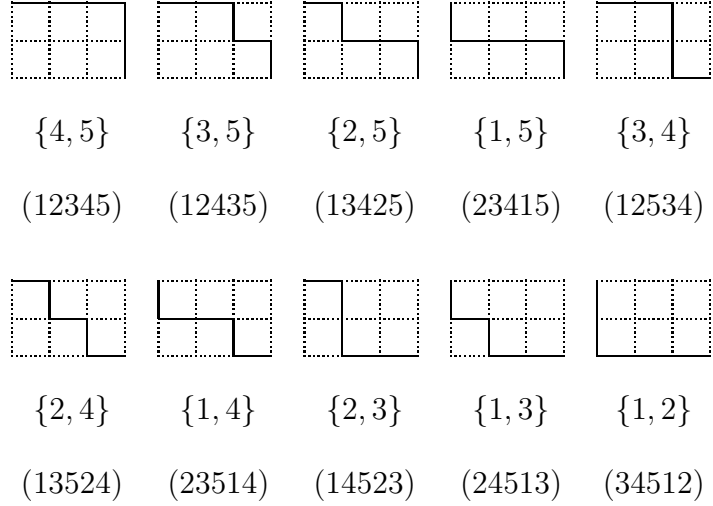
\begin{figure}
\begin{equation*}
  \begin{array}{ccccc}
    \begin{xy}
      (5,0)**@{-};(10,0)**@{-};(15,0)**@{-};
      (0,-5);(5,-5)**@{.};(10,-5)**@{.};(15,-5)**@{.};
      (0,-10);(5,-10)**@{.};(10,-10)**@{.};(15,-10)**@{.};
      (0,0);(0,-5)**@{.};(0,-10)**@{.};
      (5,0);(5,-5)**@{.};(5,-10)**@{.};
      (10,0);(10,-5)**@{.};(10,-10)**@{.};
      (15,0);(15,-5)**@{-};(15,-10)**@{-};
    \end{xy}
    \:
    &
    \begin{xy}
      (5,0)**@{-};(10,0)**@{-};(15,0)**@{.};
      (0,-5);(5,-5)**@{.};(10,-5)**@{.};(15,-5)**@{-};
      (0,-10);(5,-10)**@{.};(10,-10)**@{.};(15,-10)**@{.};
      (0,0);(0,-5)**@{.};(0,-10)**@{.};
      (5,0);(5,-5)**@{.};(5,-10)**@{.};
      (10,0);(10,-5)**@{-};(10,-10)**@{.};
      (15,0);(15,-5)**@{.};(15,-10)**@{-};
    \end{xy}
    \:
    &
    \begin{xy}
      (5,0)**@{-};(10,0)**@{.};(15,0)**@{.};
      (0,-5);(5,-5)**@{.};(10,-5)**@{-};(15,-5)**@{-};
      (0,-10);(5,-10)**@{.};(10,-10)**@{.};(15,-10)**@{.};
      (0,0);(0,-5)**@{.};(0,-10)**@{.};
      (5,0);(5,-5)**@{-};(5,-10)**@{.};
      (10,0);(10,-5)**@{.};(10,-10)**@{.};
      (15,0);(15,-5)**@{.};(15,-10)**@{-};
    \end{xy}
    \:
    &
    \begin{xy}
      (5,0)**@{.};(10,0)**@{.};(15,0)**@{.};
      (0,-5);(5,-5)**@{-};(10,-5)**@{-};(15,-5)**@{-};
      (0,-10);(5,-10)**@{.};(10,-10)**@{.};(15,-10)**@{.};
      (0,0);(0,-5)**@{-};(0,-10)**@{.};
      (5,0);(5,-5)**@{.};(5,-10)**@{.};
      (10,0);(10,-5)**@{.};(10,-10)**@{.};
      (15,0);(15,-5)**@{.};(15,-10)**@{-};
    \end{xy}
    \:
    &
    \begin{xy}
      (5,0)**@{-};(10,0)**@{-};(15,0)**@{.};
      (0,-5);(5,-5)**@{.};(10,-5)**@{.};(15,-5)**@{.};
      (0,-10);(5,-10)**@{.};(10,-10)**@{.};(15,-10)**@{-};
      (0,0);(0,-5)**@{.};(0,-10)**@{.};
      (5,0);(5,-5)**@{.};(5,-10)**@{.};
      (10,0);(10,-5)**@{-};(10,-10)**@{-};
      (15,0);(15,-5)**@{.};(15,-10)**@{.};
    \end{xy}\\
    & & & & \\
    \{4,5\} & \{3,5\} & \{2,5\} & \{1,5\} & \{3,4\}\\
    & & & & \\
    (12345) & (12435) & (13425) & (23415) & (12534)\\
    & & & & \\
    \begin{xy}
      (5,0)**@{-};(10,0)**@{.};(15,0)**@{.};
      (0,-5);(5,-5)**@{.};(10,-5)**@{-};(15,-5)**@{.};
      (0,-10);(5,-10)**@{.};(10,-10)**@{.};(15,-10)**@{-};
      (0,0);(0,-5)**@{.};(0,-10)**@{.};
      (5,0);(5,-5)**@{-};(5,-10)**@{.};
      (10,0);(10,-5)**@{.};(10,-10)**@{-};
      (15,0);(15,-5)**@{.};(15,-10)**@{.};
    \end{xy}
    \:
    &
    \begin{xy}
      (5,0)**@{.};(10,0)**@{.};(15,0)**@{.};
      (0,-5);(5,-5)**@{-};(10,-5)**@{-};(15,-5)**@{.};
      (0,-10);(5,-10)**@{.};(10,-10)**@{.};(15,-10)**@{-};
      (0,0);(0,-5)**@{-};(0,-10)**@{.};
      (5,0);(5,-5)**@{.};(5,-10)**@{.};
      (10,0);(10,-5)**@{.};(10,-10)**@{-};
      (15,0);(15,-5)**@{.};(15,-10)**@{.};
    \end{xy}
    \:
    &
    \begin{xy}
      (5,0)**@{-};(10,0)**@{.};(15,0)**@{.};
      (0,-5);(5,-5)**@{.};(10,-5)**@{.};(15,-5)**@{.};
      (0,-10);(5,-10)**@{.};(10,-10)**@{-};(15,-10)**@{-};
      (0,0);(0,-5)**@{.};(0,-10)**@{.};
      (5,0);(5,-5)**@{-};(5,-10)**@{-};
      (10,0);(10,-5)**@{.};(10,-10)**@{.};
      (15,0);(15,-5)**@{.};(15,-10)**@{.};
    \end{xy}
    \:
    &
    \begin{xy}
      (5,0)**@{.};(10,0)**@{.};(15,0)**@{.};
      (0,-5);(5,-5)**@{-};(10,-5)**@{.};(15,-5)**@{.};
      (0,-10);(5,-10)**@{.};(10,-10)**@{-};(15,-10)**@{-};
      (0,0);(0,-5)**@{-};(0,-10)**@{.};
      (5,0);(5,-5)**@{.};(5,-10)**@{-};
      (10,0);(10,-5)**@{.};(10,-10)**@{.};
      (15,0);(15,-5)**@{.};(15,-10)**@{.};
    \end{xy}
    \:
    &
    \begin{xy}
      (5,0)**@{.};(10,0)**@{.};(15,0)**@{.};
      (0,-5);(5,-5)**@{.};(10,-5)**@{.};(15,-5)**@{.};
      (0,-10);(5,-10)**@{-};(10,-10)**@{-};(15,-10)**@{-};
      (0,0);(0,-5)**@{-};(0,-10)**@{-};
      (5,0);(5,-5)**@{.};(5,-10)**@{.};
      (10,0);(10,-5)**@{.};(10,-10)**@{.};
      (15,0);(15,-5)**@{.};(15,-10)**@{.};
    \end{xy}\\
    & & & & \\
    \{2,4\} & \{1,4\} & \{2,3\} & \{1,3\} & \{1,2\}\\
    & & & & \\
    (13524) & (23514) & (14523) & (24513) & (34512)
  \end{array}
\end{equation*}
\caption{Paths, subsets and permutations}
\label{fig:1}
\end{figure}
We shall now give a combinatorial interpretation of the number $|\lambda|$.
In order to do this, we associate a permutation $\pi_\lambda$ of $(1,\ldots,n)$ (namely a rearrangement of these symbols) to each Ferrers diagram $\lambda$ contained in an $k\times n-k$ grid.
The permutation $\pi_\lambda$ is constructed as follows: look at the path corresponding to $\lambda$, with its segments indexed by the numbers $1,\ldots,n$ as before.
First write down the indices corresponding to the horizontal segments in increasing order.
Then write down the indices corresponding to the vertical segments in increasing order.

In Figure~\ref{fig:1}, for $n=5$ and $k=3$ the permutation associated to each Ferrers diagram appears below the subset corresponding to the horizontal segments.
Note that only ten of the 120 permutations of $(1,2,3,4,5)$ appear in Figure~\ref{fig:1}.
These ten permutations are characterized by the property that each entry with the possible exception of the third one is smaller than the next one.

More generally, let $(\pi(1),\ldots,\pi(n))$ denote a permutation of $(1,\ldots,n)$.
The descent set of $\pi$ is defined as
\begin{equation*}
  D(\pi)=\{ i\in \{1,\ldots,n-1\}: \pi(i)>\pi(i+1)\}.
\end{equation*}
The permutations $\pi_\lambda$ that correspond to paths in a $k\times n-k$ grid are precisely the ones for which $D(\pi)\subset \{k\}$.

We are now ready to give an interpretation of $|\lambda|$ in terms of $\pi_\lambda$.
For a permutation $\pi$ of $(1,\ldots,n)$, an inversion of $\pi$ is a pair $(i,j)$ such that $1\leq i < j \leq n$ but $\pi(i)>\pi(j)$.
Let $\inv(\pi)$ denote the number of inversions of $\pi$.
We claim that
\begin{equation}
  \label{eq:6}
  |\lambda|=\inv(\pi_\lambda).
\end{equation}
Indeed, if $(i,j)$ is an inversion of $\pi_\lambda$, then since $(\pi_\lambda(1),\ldots,\pi_\lambda(k))$ and $(\pi_\lambda(k+1),\ldots,\pi_\lambda(n))$ are increasing sequences,  we must have $1\leq i\leq k$ and $k<j\leq n$.
The box in the $i$th column and $(j-k)$th row of a Ferrers diagram $\lambda$ in a $k\times n-k$ grid contains a $*$ if and only if the associated path has taken $j-k$ downward steps before it has taken $i$ rightward steps.
But this is precisely when $\pi(i)$ (the index of the $i$th rightward step) is greater than $\pi(j)$ (the index of the $(j-k)$th downward step).
The reader who finds the above reasoning hard to follow should try at first to verify (\ref{eq:6}) for the paths in Figure~\ref{fig:1}.

By passing from $\lambda$ to $\pi_\lambda$, we can rewrite (\ref{eq:4}) in yet another form:
\begin{equation}
  \label{eq:7}
  \qbin n k = \sum_{\{\pi\in \Sigma_n:D(\pi)\subset\{k\}\}} q^{\inv(\pi)}
\end{equation}
Thus the coefficient of $q^r$ in $\qbin n k$ is the number of permutations with $r$ inversions and descent set contained in $\{k\}$.

\subsection{Multinomial coefficients}

Recall that multinomial coefficients given by the formula
\begin{equation}
  \label{eq:8}
  \multinom n{k_1&\ldots&k_m}=\frac{n!}{k_1!\cdots k_m!}
\end{equation}
count the number of ways of writing $\{1,\ldots,n\}$ as a disjoint union $X_1\sqcup\cdots\sqcup X_m$ where $|X_i|=k_i$ for $i=1,\ldots,m$ (of course, one assumes that $k_1+\cdots+k_m=n$).
Let
\begin{equation*}
  \begin{aligned}
    Y_1&=X_1&\\
    Y_2&=X_1\cup X_2&\\
    &\vdots&\\
    Y_{m-1}&=X_1\cup\cdots\cup X_{m-1}&\\
    Y_m&=X_1\cup\cdots\cup X_m&=\{1,\ldots,n\}
  \end{aligned}
\end{equation*}
Let $s_i=k_1+\cdots+k_i$ for each $i\in \{1,\ldots,m-1\}$.
Then $Y_1\subset Y_2\subset\cdots\subset Y_{m-1}$ is a nested sequence of subsets such that $|Y_i|=s_i$ for all $i\in \{1,\ldots,m-1\}$, and it is clear that the sequence of disjoint sets $(X_i)$ can be recovered from the sequence of nested sets $(Y_i)$.
Writing $S$ for an increasing sequence $(s_1,\ldots,s_{m-1})$ in $\{1,\ldots,n\}$, we shall use the notation
\begin{equation*}
  \binom n S=\#\big\{S_1\subset\cdots\subset S_m= \{1,\ldots,n\}:|S_i|=s_i \text{ for } 1\leq i<m \big\}.
\end{equation*}
We have:
\begin{equation*}
  \binom n S=\multinom n{k_1&\cdots&k_m},
\end{equation*}
where $k_i=s_i-s_{i-1}$ for $1,i< m$, $k_1=s_1$ and $k_m=n-s_{m-1}$.
This interpretation of the multinomial coefficients has an obvious analogue for vector spaces and subspaces.
Define
\begin{equation*}
  \binom n S_q = \#\{V_1\subset \cdots\subset V_{m-1}\subset F^n: V_i \text{ is a subspace of dimension }s_i\}.
\end{equation*}
The multinomial coefficients can be expressed in terms of the binomial coefficients.
Clearly, the number of choices for $S_{n-1}$ inside $S_n$ is $\binom n{s_{n-1}}$, the number of choices for $S_{n-2}$ inside each choice of $S_{n-1}$ is $\binom{s_{n-1}}{s_{n-2}}$ and so on, so that
\begin{equation}
  \label{eq:9}
  \binom n S = \binom{s_2}{s_1}\cdots\binom{s_{m-1}}{s_{m-2}}\binom n{s_{m-1}}.
\end{equation}
Similarly,
\begin{equation}
  \label{eq:10}
    \qbin n S = \qbin{s_2}{s_1}\cdots\qbin{s_{m-1}}{s_{m-2}}\qbin n{s_{m-1}}.
\end{equation}
These expressions allow us to easily deduce some properties of $q$-multinomial coefficients from the corresponding properties of $q$-binomial coefficients.
For example, the expression (\ref{eq:5}) expressing a $q$-binomial coefficient as a rational function in $q$ with denominator non-zero except at $q=1$ also shows that a $q$-multinomial coefficient is a rational function of $q$ with denominator non-zero except at $q=1$.
From the identity~(\ref{eq:3}) we can deduce that
\begin{equation}
  \label{eq:multinom}
  \lim_{q\to 1}\qbin n S=\binom n S.
\end{equation}
Similarly, from the identity (\ref{eq:4}) we can deduce that each $q$-multinomial coefficient is a monic polynomial in $q$ of degree $k_1\cdots k_m$ with positive coefficients.
A slightly more careful analysis will allow us to find the analogue of (\ref{eq:7}) for multinomial coefficients.
Not only will it give us a combinatorial interpretation for the coefficient of a power of $q$ in the multinomial coefficient,
when combined with the principle of inclusion and exclusion, it will also have the rather surprising consequence that the alternating sum
\begin{equation*}
  \sum_{S\subset T} (-1)^{|T-S|}\qbin n S
\end{equation*}
(where $|T-S|$ denotes the number of elements of $T$ which are not in $S$) is a polynomial in $q$ with non-negative coefficients.

In order to simplify notation, for any $S\subset \{1,\ldots,n\}$, let
\begin{equation*}
  \Sigma_n(S)=\{\pi \in \Sigma_n \;:\; D(\pi)\subset S\}.
\end{equation*}
Then (\ref{eq:7}) can be rewritten as
\begin{equation*}
  \qbin n k = \sum_{\pi \in \Sigma_n(\{k\})} q^{\inv(\pi)}.
\end{equation*}
If we expand the right hand side of (\ref{eq:10}) using the above identity, we get
\begin{equation}
  \label{eq:12}
  \sum_{\pi_1\in\Sigma_{s_2}(\{s_1\})}
  \:\:
  \cdots
  \:\:
  \sum_{\pi_{m-2}\in \Sigma_{s_{m-1}}(\{s_{m-2}\})}
    \quad
  \sum_{\pi_{m-1}\in \Sigma_n(\{s_{m-1}\})}
    q^{\inv(\pi_1)+\cdots+\inv(\pi_m)}
\end{equation}
We claim that the formula
\begin{equation*}
  \Phi(\pi_{m-1},\pi_{m-2},\ldots,\pi_1)=\pi_{m-1}\circ\pi_{m-2}\circ\cdots\circ\pi_1
\end{equation*}
defines a bijection
\begin{equation*}
  \Phi:\Sigma_{s_2}(\{s_1\})\times \Sigma_{s_3}(\{s_2\})\times \cdots \times \Sigma_{s_{m-1}}(\{s_{m-1}\})\tilde\to\Sigma_n(S)
\end{equation*}
(where $\Sigma_r$ is identified with the subgroup of $\Sigma_{n}$ which permutes the first $r$ elements of $\{1,\ldots,n\}$) 
such that 
\begin{equation*}
  \inv(\Phi(\pi_{m-1},\pi_{m-2},\ldots,\pi_1))=\inv(\pi_{m-1})+\inv(\pi_{m-2})+\cdots+\inv(\pi_1).
\end{equation*}
allowing us to rewrite  (\ref{eq:12}) as
\begin{equation*}
  \qbin n S = \sum_{\pi\in \Sigma_n(S)} q^{\inv(\pi)}.
\end{equation*}
Indeed, using induction on $m$, it suffices to verify
\begin{lemma*}
  The formula
  \begin{equation*}
    (\pi',\pi_{m-1})\mapsto \pi'\circ \pi_{m-1}
  \end{equation*}
  defines a bijection
  \begin{equation*}
    \Sigma_{s_{m-1}}(\{s_1,\ldots,s_{m-2}\})\times \Sigma_n(\{s_{m-1}\})\to \Sigma_n(S)
  \end{equation*}
  such that
  \begin{equation}
    \label{eq:13}
    \inv(\pi_{m-1}\circ\pi')=\inv(\pi_{m-1})+\inv(\pi')
  \end{equation}
  (where $\Sigma_{s_{m-1}}$ is identified with the subgroup of $\Sigma_n$ which permutes the first $s_{m-1}$ elements of $\{1,\ldots,n\}$).
\end{lemma*}
\begin{proof}
  Firstly, note that if $\pi'\in\Sigma_{s_{m-1}}(\{s_1,\ldots,s_{m-2}\})$ and $\pi_{m-1}\in\Sigma_n(\{s_{m-1}\})$ then both $\pi'$ and $\pi_{m-1}$ are increasing functions on each of the segments $\{s_{r-1}+1,\ldots,s_r\}$ for $r=1,\ldots,m$, and therefore, their composition is also increasing on these segments.
  Therefore $\pi'\circ\pi_{m-1}\in \Sigma_n(S)$.

  It remains to prove (\ref{eq:13}).
  For each permutation $\pi$, let $I(\pi)$ denote the set of inversions of $\pi$, namely, the set of pairs $i<j$ for which $\pi(i)>\pi(j)$.
  In general, if $\sigma$ and $\tau$ are permutations then
  \begin{equation*}
    I(\sigma\circ\tau)\subset \tau^{-1}(I(\sigma))\cup I(\tau)
  \end{equation*}
  from which it follows that
  \begin{equation*}
    \inv(\sigma\circ\tau)\leq\inv(\sigma)+\inv(\tau)
  \end{equation*}
  If in addition, $\tau^{-1}(I(\sigma))\cap I(\tau)=\emptyset$, then 
  \begin{equation*}
    I(\sigma\circ\tau)= \tau^{-1}(I(\sigma))\cup I(\tau)
  \end{equation*}
  and therefore,
  \begin{equation*}
    \inv(\sigma\circ\tau)=\inv(\sigma)+\inv(\tau)
  \end{equation*}
  Hence in order to prove (\ref{eq:13}), one must show that
  \begin{equation}
    \label{eq:14}
    \pi^{\prime-1}(I(\pi_{m-1}))\cap I(\pi')=\emptyset.
  \end{equation}
  Since $\pi'$ acts only on $\{1,\ldots,s_{m-1}\}$, if $(i,j)$ is an inversion of $\pi'$ then $\pi'(j)<\pi'(i)\leq s_{m-1}$.
  Since $\pi_{m-1}$ has no descents in this range, $(\pi'(i),\pi'(j))$ can not be an inversion of $\pi_{m-1}$, proving (\ref{eq:14}).
\end{proof}
We have succeeded in generalizing (\ref{eq:7}) to multinomial coefficients:
\begin{equation}
  \label{eq:15}
  \qbin n S = \sum_{D(\pi)\subset S} q^{\inv(\pi)}.
\end{equation}
Combining (\ref{eq:multinom}) with (\ref{eq:15})
\begin{equation}
  \label{eq:11}
  \binom n S = \#\{\pi\in \Sigma_n|D(\pi)\subset S\}
\end{equation}
We shall now combine these identity with the principle of inclusion and exclusion to get a surprising positivity result for alternating sums of multinomial coefficients.

\subsection{Principle of inclusion and exclusion}
\label{sec:princ-incl-excl}


This principle is most commonly known as a technique for computing the cardinality of a union of finite sets in terms of cardinalities of various intersections. 
What follows is a more abstract formulation (see \cite[Section~2.1]{MR1442260}).

Let $R$ be a finite set.
The power set $2^R$ of $R$ is the set of all subsets of $R$.
The principle of inclusion and exclusion can be interpreted as saying that two functions $\alpha,\beta:2^R\to \C$ satisfy the identities
\begin{equation}
  \tag{$\spadesuit$}
  \label{eq:1}
  \beta(T)=\sum_{S\subset T} \alpha(S).
\end{equation}
for all $S,T\subset R$ if and only if they satisfy the identities
\begin{equation}
  \tag{$\clubsuit$}
  \label{eq:2}
  \alpha(T)=\sum_{S\subset T} (-1)^{|T-S|} \beta(S)
\end{equation}
for all $S,T\subset R$.
\begin{proof}
  Assume (\ref{eq:1}). 
  Then
  \begin{eqnarray*}
    \sum_{S\subset T}(-1)^{|T-S|}\beta(S) & = & \sum_{S\subset T}(-1)^{|T-S|}\sum_{U\subset S} \alpha(U)\\
    & = & \sum_{U\subset T}\alpha(U)\sum_{U\subset S \subset T} (-1)^{|T-S|}
  \end{eqnarray*}
  If $m=|T-U|>0$, then the inner sum is $\sum_{i=0}^m (-1)^i\binom{m}{i}=0$.
  Therefore, the expression reduces to $\alpha(T)$.
  Conversely, assume (\ref{eq:2}).
  Then
  \begin{eqnarray*}
    \sum_{S\subset T} \alpha(T) & = & \sum_{S\subset T}\sum_{U\subset S}(-1)^{|S-U|}\beta(U)\\
    & = & \sum_{U\subset T} \beta(U) \sum_{U\subset S\subset T} (-1)^{|S-U|},
  \end{eqnarray*}
  and as before, the inner sum is zero unless $U=T$, in which case it is one.
\end{proof}

\subsection{Counting permutations with a given descent set}
\label{sec:count-perm-with}

The principle of inclusion and exclusion allows us to use (\ref{eq:11}) to write a formula for the number of permutations with a given descent set in terms of multinomial coefficients.
Take $R=\{1,\ldots,n\}$ and let
\begin{equation*}
  \alpha(S)=\#\{\pi\in \Sigma_n|D(\pi)=S\}.
\end{equation*}
Then by (\ref{eq:11})
\begin{equation*}
  \beta(T)=\sum_{S\subset T}\alpha(S)=\#\{\pi\in \Sigma_n|D(\pi)\subset T\}=\binom n T.
\end{equation*}
Thus, by the principle of inclusion and exclusion
\begin{equation*}
  \#\{\pi\in \Sigma_n|D(\pi)=T\}=\sum_{S\subset T}(-1)^{|T-S|}\binom n S.
\end{equation*}

\subsection{A surprising positivity result}

Let
\begin{equation*}
  \alpha_q(S)=\sum_{D(\pi)=S} q^{\inv(\pi)}.
\end{equation*}
Then
\begin{eqnarray*}
  \beta_q(T)&=&\sum_{S\subset T}\alpha_q(S)\\
  &=&\sum_{S\subset T}\sum_{D(\pi)=S}q^{\inv(\pi)}\\
  &=&\sum_{D(\pi)\subset T}q^{\inv(\pi)}\\
  &=&\qbin n T
\end{eqnarray*}
The principle of inclusion and exclusion (Section~\ref{sec:princ-incl-excl}) gives
\begin{equation}
  \label{eq:16}
  \sum_{S\subset T}(-1)^{|T-S|}\qbin n S=\sum_{D(\pi)=T}q^{\inv(\pi)}.
\end{equation}
Not only is this alternating sum positive; in fact it is a polynomial in $q$ with positive coefficients!

\subsection{Closing remarks}

The combinatorics of Gaussian binomial coefficients continues to fascinate mathematicians today.
That the binomial coefficients $\binom n k$ increase with $k$ when $k\leq n/2$ and decrease for $k\geq n/2$ is easy to see.
An analogous statement for Gaussian binomial coefficients is also not too difficult.
Such sequences of numbers are called \emph{unimodal}.
Towards the end of Section~\ref{sec:echelon}, we saw that $\qbin nk$ is a polynomial in $q$ with positive coefficients.
The first combinatorial proof of the unimodality of these coefficients was found as late as 1990 by Kathleen M. O'Hara \cite{O'Hara199029} (see also an expository article on O'Hara's proof by Doron Zeilberger \cite{Zeilberger}).

Counting subspaces in a finite dimensional vector space over the field $\Z/p\Z$ is a special case of the problem of counting subgroups inside a finite abelian group.
Every finite abelian group can be written as a product of over a finite set of prime numbers of groups of the form
\begin{equation}
  \label{eq:17}
  A_{p,\lambda}=\Z/p^{\lambda_1}\Z\times \cdots \times \Z/p^{\lambda_l}\Z
\end{equation}
where $\lambda=(\lambda_1\geq\cdots\geq\lambda_l)$ is a sequence of positive integers.

A generalization of the problem of counting $k$-dimensional subspaces in an $n$-dimensional space is the problem of counting subgroups of type $A_{p,\mu}$ inside $A_{p,\lambda}$ for $\mu=(\mu_1\geq\cdots\geq\mu_m)$.
The generalizations of (\ref{eq:5}) and (\ref{eq:3}) have been known for a long time (e.g., Delsarte \cite{MR0025463}) and are related to multiset combinatorics (a multiset is a set where elements are allowed to appear with multiplicities, for example, as in the case of roots of a polynomial).
However, the analogues of (\ref{eq:15}) and (\ref{eq:16}) are more subtle, and can be found in Lynne Butler's beautiful monograph \cite{MR1223236}.

In recent years, techniques from combinatorics have been combined with those from analytic number theory (namely, zeta functions) to study enumeration problems in algebra with great success (see, for example the survey article \cite{ZetaICM06}).
The reader who wishes to start learning the language and basic techniques of modern combinatorics need look no further than Richard P. Stanley's amazing book \cite{MR1442260}.
A slightly different take on some of the contents of this article can be found in Henry Cohn's expository article \cite{CohnAMM}.

\subsection*{Acknowledgements}

I benefited from discussions with Kunal Dutta, Sanoli Gun and Shailesh Shirali while preparing this article. It is a pleasure to thank them.

\end{document}